\documentclass[11pt]{article}
\usepackage{amsfonts}
\usepackage{amscd}
\usepackage{amsmath}
\usepackage{epsfig}
\usepackage{amsthm}
\usepackage{amssymb}
\usepackage{amsbsy}
\usepackage{latexsym}
\usepackage{eucal}
\usepackage{mathrsfs}
\usepackage[all, knot]{xy}
\xyoption{arc} \xyoption{color}\xyoption{line}
\usepackage{color}

\newtheorem{theorem}{Theorem}
\newtheorem{proposition}{Proposition}

\newtheorem{remark}{Remark}
\newtheorem{lemma}{Lemma}

\def\be{\begin{equation}}
\def\ee{\end{equation}}
\def\ben{\begin{displaymath}}
\def\een{\end{displaymath}}
\def\baa{\begin{eqnarray}}
\def\eaa{\end{eqnarray}}

\def\ba{\begin{array}}
\def\ea{\end{array}}
\renewcommand{\leq}{\leqslant}
\renewcommand{\geq}{\geqslant}

\newcommand{\supp}{\operatorname{supp}}

\newcommand{\Tr}{\operatorname{Tr}}
\newcommand{\Res}{\operatorname{Res}}

\newcommand{\Id}{\operatorname{Id}}
\renewcommand{\det}{\operatorname{det}}

\makeatletter
\@addtoreset{equation}{section}
\makeatother

\begin{document}
\title {Metrics of constant positive curvature with conical singularities, Hurwitz spaces,  and ${\rm det}\, \Delta$ }

\author{Victor Kalvin \footnote{{\bf E-mail: victor.kalvin@concordia.ca, vkalvin@gmail.com}}, Alexey Kokotov\footnote{{\bf E-mail: alexey.kokotov@concordia.ca}} }
\date{}
\maketitle

\vskip0.5cm
\begin{center}
Department of Mathematics and Statistics, Concordia
University, 1455 de Maisonneuve Blvd. West, Montreal, Quebec, H3G
1M8 Canada \end{center}

\vskip2cm
{\bf Abstract.} Let $f: X\to {\Bbb  C}P^1$ be a meromorphic function of degree $N$  with simple poles and simple critical points on a compact Riemann surface $X$ of genus $g$ and let $\mathsf m$ be the standard round metric of curvature $1$ on the Riemann sphere
${\Bbb  C}P^1$.     Then the pullback $f^*\mathsf m$ of $\mathsf m$ under $f$  is a metric of curvature $1$ with conical singularities of conical angles $4\pi$ at the critical points of $f$.    We study the $\zeta$-regularized determinant of the Laplace operator on $X$ corresponding to the  metric $f^*\mathsf m$ as a functional on the moduli space of the pairs $(X, f)$  (i.e. on the  Hurwitz space $H_{g, N}(1, \dots, 1)$) and derive an explicit formula for the functional.

\vskip2cm

\section{Introduction}\label{intro}

The determinants of Laplacians on Riemann surfaces often appear in the frameworks of Geometric Analysis (in connection with Sarnak program \cite{Sarnak}) and quantum field theory (in connection with various partition functions).
An explicit computation of the determinant of the Laplacian corresponding to the metric of constant negative curvature (\cite{DPh}, see also \cite{Fay92}) provides an example of a beautiful interplay between the spectral theory and geometry of moduli spaces of Riemann surfaces. Due to Gauss-Bonnet Theorem metrics of constant positive curvature on compact Riemann surfaces are necessarily singular (unless the genus of the surface is equal to zero) and the same is true for metrics of zero curvature (unless the genus is equal to one). The determinants of the Laplacians in flat singular metrics are intensively studied (see, e. g., \cite{Khuri}, \cite{Au-Sal}, \cite{KK-DG}, \cite{ProcAMS}, \cite{KKH-Communications}), the case of constant positive curvature attracted attention only recently (in particular, in connection with quantum Hall effect). The only explicit computation of the determinant in the case of constant positive curvature (except for the classical result for the smooth round metric on the sphere \cite{Weisberger}) is done in the case of the sphere with two antipodal conical singularities (\cite{Sp}, see also \cite{Klevtsov} for corrections and  a relation of this result to quantum physics). 
According to the result of Troyanov \cite{Troyanov}, there are only two classes of genus zero surfaces with metrics of constant curvature $1$ with two conical points:
\begin{itemize}
\item Surfaces with two {\it antipodal}  conical singularities (i.e.  the distance  between them is $\pi$ and they are conjugate points) of the same (arbitrary positive) conical angle.

\item Surfaces with two conical points of the same angle $2\pi k$, $k=2,3, \dots$;  the corresponding conical metric is the pullback $f^*\mathsf m$ of the standard metric $\mathsf m$ of curvature $1$ on $\mathbb CP^1$ under a
meromorphic function
$f:{\mathbb C}P^1\to {\mathbb C}P^1$
with two critical points.
\end{itemize}

As we already mentioned, the determinant of the Laplacian on the surfaces of the first class was found in \cite{Sp, Klevtsov}.  The  motivation of this paper comes mainly from the need to compute the determinant of the Laplacian $\Delta^{f^*\mathsf m}$ for the surfaces of the second class. For this determinant we obtain the explicit formula 
\begin{equation}\label{eq0}{\operatorname{det}}'\, \Delta^{f^*\mathsf m}=C|z_1-z_2|^{\frac{1}{2}}(1+|z_1|^2)^{-\frac{1}{4}}(1+|z_2|^2)^{-\frac{1}{4}}\,,\end{equation}
which is  the most elementary consequence of our main result. Here $f:{\mathbb C}P^1\to {\mathbb C}P^1$ is a meromorphic function with two simple critical points and the corresponding critical values $z_1$ and $z_2$, the constant $C$ is independent of $z_1$ and $z_2$, and ${\operatorname{det}}'$ is the modified (i.e. with zero mode excluded) $\zeta$-regularized determinant. The constant $C$ can be  found by using the result~\cite{Sp}: one has to consider a sphere with two antipodal singularities of conical angle $4\pi$
and compare formula (\ref{eq0}) with the one given in \cite{Sp}).

Our main result generalizes (\ref{eq0}) to the case of compact Riemann surfaces $X$ of arbitrary genus and arbitrary meromorphic functions $f: X\to {\mathbb C}P^1$ (for simplicity we consider only functions $f$ with simple critical values, the modifications required to consider the general case are insignificant and of no interest, and the result is essentially the same).

Let $H_{g, N}(1, \dots, 1)$ be the Hurwitz moduli space of pairs $(X, f)$, where $X$ is a compact Riemann surface of genus $g$ and $f$ is a meromorphic function on $X$ of degree $N$ and $M=2g-2+2N$ simple critical points. We assume that all the critical values are finite, i.e.  the poles of the function $f$ are not the critical points and, therefore, are simple. The  part $(1, \dots, 1)$ ($N$ times) of the symbol $H_{g, N}(1, \dots, 1)$ shows the branching scheme over the point at infinity of the base of the ramified covering
$f: X\to {\mathbb C}P^1$,
the preimage of $\infty\in {\mathbb C}P^1$ consists of $N$ distinct points.
The space $H_{g, N}(1, \dots, 1)$  is known to be a connected complex manifold of complex dimension $M$, the critical values $z_1, \dots, z_M$ of the function $f$ can be taken as local coordinates.

Let $\tau$ stand for the Bergman tau-function on the Hurwitz space $H_{g, N}(1, \dots, 1)$ (also known as isomonodromic tau-function of the Hurwitz Frobenius manifold). Referring the reader to \cite{KK},  \cite{KK1}, \cite{KKZ} for definition and properties of this object, we would like to emphasize that the explicit expressions for $\tau$ through holomorphic invariants of the Riemann surface (prime form, theta functions, and etc.) and the divisor of the meromorphic function $f$ are known; see \cite{KK1,KS} for genera $g=0,1$ and~\cite{KK,KKZ} for $g\geq 2$.

The metric $f^*{\mathsf m}$ on $X$ is a conical metric of curvature $1$ with conical singularities at the critical points $P_1, \dots, P_M$ of the function $f$, the  conical angle at any critical point is $4\pi$.

In the present paper we first show that the operator zeta-function $\zeta(s)$ of the Friedrichs extension of the Laplace operator $\Delta$ 
is regular at the point $s=0$ and, therefore, one can define the (modified, i.e. with zero mode excluded) $\zeta$-regularized determinant
$$ {\rm det}'\, \Delta^{f^*{\mathsf m}}:=\exp\{-\zeta'(0)\}\,.$$
Then we prove the following explicit formula for this determinant:
\begin{equation}\label{teorema-p}
{\rm det}'\Delta=C\,{\rm det}\, \Im {\mathbb B} |\tau|^2 \prod_{k=1}^M(1+|z_k|^2)^{-1/4}\,,
\end{equation}
where the constant $C$ is independent of the
point $(X, f)$ of the space  $H_{g, N}(1, \dots, 1)$ and ${\mathbb B}$ is the matrix of $b$-periods of the Riemann surface $X$ (in the case $g=0$ the factor ${\rm det}\, \Im {\mathbb B}$ in (\ref{teorema-p}) should be omitted).
In the simplest case one has $g=0$, $N=2$, and  $\tau=(z_1-z_2)^{1/4}$, then~\eqref{teorema-p} implies \eqref{eq0}.

{\bf Acknowledgements.} We would like to thank Luc Hillairet for fruitful discussions and for communicating some important ideas.
We are extremely grateful to  Semyon Klevtsov and Paul Wiegmann for stimulating questions one of which gave rise to the present work.
We also thank Dmitry Korotkin for several advices. The research of the second author was supported by NSERC.

\section{Heat kernel asymptotic and $\det'\Delta$}

Let  $\Delta$ stand for the Friedrichs extension  of  the Laplace-Beltrami operator on $(X, f^*\mathsf m)$.  The asymptotic of $\Tr e^{-\Delta t}$ as $t\to 0+$ can be found by methods developed in~\cite{BS1, BS2, Seeley}.
We need some preliminaries before we can formulate the result.

Introduce the local geodesic polar coordinates $(r,\varphi)$ on $(X,f^*\mathsf m )$ with center at $P_k$, where $\varphi\in [0,4\pi)$ and $r\in[0,\epsilon]$, $\epsilon$ is smaller than the distance from $P_k$ to any other conical singularity. 
In the coordinates $(r,\varphi)$ the metric $f^*\mathsf m$ takes the form
$$
f^*\mathsf m (r,\varphi)= dr^2+  \sin ^2 r d\varphi^2.
$$
 Let  $h(r)=2\sin r$ and $\psi=\varphi/2\in \Bbb S^1$. Consider  the selfadjoint operator
\begin{equation}\label{e1}
\mathcal A(r)=- r^2h^{-2}(r)\partial^2_\psi- r^2 \bigl(\cot^2 r+2\bigr)/4,\quad r\in[0,\epsilon],
\end{equation}
  in $L^2(\Bbb S^1)$  with the domain $H^2(\Bbb S^1)$.  This operator is related to $\Delta$ in the following way:  In a small neighbourhood of $P_k$ the Laplacian can be written as
  $$
  \Delta =  h^{-1/2}(-\partial_r^2+r^{-2} \mathcal A(r))h^{1/2}  $$ acting  in  $L^2(h\,dr\,d\psi)$. The  operator $L=-\partial_r^2+r^{-2} \mathcal A(r)$  falls into the class of operators studied in~\cite{BS2} as $\mathcal A(r)$ satisfies the requirements~\cite[(A1)--(A6), page 373]{BS2}. 
  Then \cite[Thm~5.2 and Thm 7.1]{BS2} imply that for any smooth cut-off function $\varrho$ supported sufficiently close to the singularity $P_k$ and such that $\varrho=1$ in a small vicinity of $P_k$ one has
\begin{equation}\label{HA}
\Tr \varrho e^{-\Delta t}\thicksim\sum_{j=0}^\infty A_j t^{\frac {j-3}{2}} +\sum_{j=0}^\infty B_j t^{-\frac{\alpha_j+4}{2}}+\sum_{j: \alpha_j\in\Bbb Z_- } C_j  t^{-\frac{\alpha_j+4}{2}} \log t\quad \text{    as } t\to0+,
\end{equation}
where $A_j$, $B_j$, and $C_j$ are some coefficients and $\{\alpha_j\}$ is a sequence of complex numbers with $\Re\alpha_j\to-\infty$. Moreover,  the coefficient  before $t^0\log t$ in the above asymptotic is given by
$\frac1 4\Res \zeta(-1)$, where $\zeta$ stands for the $\zeta$-function of $(\mathcal A(0)+1/4)^{1/2}$; see \cite[f-la (7.24)]{BS2}. Clearly, $\mathcal A(0)=-2^{-2}\partial^2_\psi-1/4$ and the $\zeta$-function of  $( {\mathcal A(0)+1/4})^{1/2}$ is given by 

$$\zeta(s)=2\sum_{j\geq1} (j/2)^{-s}=2^{s+1}\zeta_R(s),$$
where $\zeta_R$ is the Riemann zeta function. Thus $\Res \zeta(-1)=0$ and the term with $t^0\log t$ in~\eqref{HA} is absent.

For a cut-off function $\rho$ supported outside of conical points $P_1,\dots,P_M$ the short time asymptotic  $\Tr(1-\rho) e^{-\Delta t}\sim \sum_{j\geq -2} a_jt^{j/2} $ can be obtained in the standard way from the formulas for the parametrix $B^N (\lambda)$ approximating $(\Delta-\lambda)^{-2}$
 to the order $N$, see e.g.~\cite{Seeley} or~\cite[Problem 5.1]{Shubin}. Hence the short time asymptotic for $e^{-\Delta t}$ is of the form~\eqref{HA}, where the term $t^0\log t$ is absent. As a consequence, the $\zeta$-function $$\zeta(s)=\frac{1}{\Gamma(s)}\int_0^\infty t^{s-1}(\Tr e^{-t\Delta} -1)\,dt $$ has no pole at zero and we can define the modified (i.e. with zero mode excluded)  determinant $\det' \Delta=\exp\{-\zeta'(0)\}$. 
\section{Asymptotic of solutions near conical singularities}

In a vicinity of  $P_k$ we introduce  the {\it distinguished} local parameter $x=\sqrt{z-z_k}$ . Since
\begin{equation}\label{m}
\mathsf m=\frac{4|dz|^2}{(1+|z|^2)^2},
\end{equation}
 we have
\begin{equation}\label{lap}
f^*\mathsf m({x},\bar x)=\frac {16|{x}|^2\,|d{x}|^2}{(1+|{x}^2+z_k|^2)^2}\quad\text{and}\quad \Delta^*= -\frac {(1+|{x}^2+z_k|^2)^2} {4|{x}|^2}  \partial_{x}\partial_{\bar {x}}.
\end{equation}
Here and elsewhere we denote the Laplace-Beltrami operators by $\Delta^*$ reserving the notation $\Delta$ for their Friederichs extensions. The complex plane $\Bbb C$ endowed with the metric $f^*\mathsf m({x},\bar x)$ has a ``tangent cone'' of angle $4\pi$ at ${x}=0$.

\begin{lemma}\label{expansion}Let $u,F\in L^2(X)$ and $\Delta^* u=F$ (in the sense of distributions). 
Then in a small vicinity of ${x}=0$ we have
\begin{equation}\label{EU}
u({x},\bar {x})=a_{-1}\bar {x}^{-1} +b_{-1}{x}^{-1}+  a_0\ln|{x}|+ b_0+ a_1\bar {x} +b_1 {x}+ R({x},\bar {x}),
\end{equation}
where  $a_k$ and $b_k$ are some coefficients and the remainder $R$ satisfies $R({x},\bar {x})=O(|{x}|^{2-\epsilon})$ with  any $\epsilon>0$  as $x\to 0$.  Moreover,  the equality  can be differentiated and   the remainder satisfies $\partial_{x} R({x},\bar {x})=O(|{x}|^{1-\epsilon})$ and $\partial_{\bar {x}} R({x},\bar {x})=O(|{x}|^{1-\epsilon})$ with  any $\epsilon>0$.
\end{lemma}
\begin{proof} The proof consists of standard steps based on the Mellin transform and a priori elliptic estimates, see e.g.~\cite[Chapter 6]{KMR} for details. 

Let $\chi\in C^\infty_c(X)$ be a cut-off function supported in the neighbourhood $|x|<2\delta$ of $P_k$ and  such that $\chi(|x|)=1$ for $|x|<\delta$, where $\delta$ is small.  Then $\Delta^* u=F$ implies
\begin{equation}\label{E1}
-|x|^{-2}\partial_x\partial_{\bar x}\bigl(\chi u\bigr)(x,\bar x)=4\frac{\bigl(\chi F\bigr)(x,\bar x)+[\Delta^*,\chi]u(x,\bar x)}{{(1+|{x}^2+z_k|^2)^2} },
\end{equation}
where the right hand side (extended from its support to $X$ by zero) is in $L^2(X)$. Indeed, for any cut-off function $\varrho\in C_c^\infty(X\setminus\{P_1,\dots, P_M\})$ the standard result on smoothness of solutions to elliptic problems gives $\varrho u\in H^1(X)$, where the Sobolev space $H^1(X)$ is the domain of closed densely defined quadratic form of $\Delta^*$ in $L^2(X)$. For a suitable $\varrho$ we obtain $[\Delta^*,\chi]u=[\Delta^*,\chi]\varrho u\in L^2(X)$ and hence the right hand side of~\eqref{E1}  is in $L^2(X)$.

We rewrite~\eqref{e1} in the polar coordinates $(r,\varphi)$, where $r=|x|^2$ and $\varphi=\arg x$,  multiply both sides by $r^2$, and then apply the  Mellin transform $\hat f(s)= \int_0^\infty r^{s-1}f(r)\,dr$, assuming that all functions are extended from their supports to $r\in [0,\infty)$ and $\varphi\in [0,2\pi)$ by zero. As a result~\eqref{e1} takes the form $ -\bigl(4^{-1}\partial_\varphi^2-s^2\bigr)\widehat{\chi u}(s) = \widehat{G}(s)$. 
  Due to the inclusion $u\in L^2(X)$ (resp. $r^{-2}G\in L^2(X)$) the function $s\mapsto \widehat{\chi}(s)\in L^2(\Bbb S^1)$  (resp. $s\mapsto \widehat{G}(s)\in L^2(\Bbb S^1)$) is analytic in the half-plane $\Re s>1$ (resp. $\Re s>-1 $)  and square summable along any vertical line in the corresponding half-plane. In the strip $-1<\Re s<1$ the resolvent $\bigl(4^{-1}\partial_\varphi^2-s^2\bigr)^{-1}: L^2(\Bbb S^1)\to L^2(\Bbb S^1)$ has simple poles at $s=\pm1/2$ and a double pole at $s=0$. We have
$$
\bigl(\chi u\bigr)(r)=\frac{1}{2\pi i}\int_{1-\epsilon-i\infty}^{1-\epsilon+i\infty} r^{-s}\widehat{\chi u}(s)\,ds=-\frac{1}{2\pi i}\int_{1-\epsilon-i\infty}^{1-\epsilon+i\infty} r^{-s}\bigl(4^{-1}\partial_\varphi^2-s^2\bigr)^{-1}\widehat{ G}(s)\,ds,
$$
where $\epsilon\in (0,1/2)$. The elliptic a priori estimate with parameter
\begin{equation}\label{EP}
\begin{aligned}
\sum_{\ell=0}^2 |s|^{2\ell}\Bigl\|\partial_\varphi^{2-\ell}\Bigl\{ \bigl(4^{-1}\partial_\varphi^2&-s^2\bigr)^{-1}\widehat G(s)\Bigr\}; L^2(\Bbb S^1)\Bigr\|^2\\ &\leq  C\Bigl(\|\widehat G(s); L^2(\Bbb S^1)\|^2+\| \bigl(4^{-1}\partial_\varphi^2-s^2\bigr)^{-1}\widehat G(s); L^2(\Bbb S^1)\|^2\Bigr),
\end{aligned}
\end{equation}
where the last term can be neglected  for sufficiently large values of $|s|$, justifies  the change of the contour of integration in the inverse Mellin transform from $\Re s=1-\epsilon$ to $\Re s=-1+\epsilon$.
We use the Cauchy theorem and arrive at
\begin{equation*}\label{srez}
\bigl(\chi u\bigr)({x},\bar {x})=a_{-1}\bar {x}^{-1} +b_{-1}{x}^{-1}+  a_0\ln|{x}|+ b_0+ a_1\bar {x} +b_1 {x}+ R({x},\bar {x}),
\end{equation*}
where
$$
 R({x},\bar {x})=R(r,\varphi)=-\frac{1}{2\pi i}\int_{-1+\epsilon-i\infty}^{-1+\epsilon+i\infty} r^{-s}\bigl(4^{-1}\partial_\varphi^2-s^2\bigr)^{-1}\widehat{ G}(s)\,ds.
$$
The boundedness of  $s\mapsto\bigl(4^{-1}\partial_\varphi^2-s^2\bigr)^{-1}$ on the line $\Re s=-1+\epsilon$ and~\eqref{EP} give

\begin{equation}\label{estS}
\sum_{\ell=0}^2 (1+|s|^2)^\ell\Bigl\|\partial_\varphi^{2-\ell}\Bigl\{ \bigl(4^{-1}\partial_\varphi^2-s^2\bigr)^{-1}\widehat G(s)\Bigr\}; L^2(\Bbb S^1)\Bigr\|^2\leq C\|\widehat G(s); L^2(\Bbb S^1)\|^2,
\end{equation}
 where $C$ does not depend on $s$.  The Parseval equality turns~\eqref{estS}  into the estimate 
$$
\begin{aligned}
\int_0^{2\pi}\int_0^{2\delta}r^{-4+2\epsilon}\Bigl( \sum_{\ell=0}^2|(r\partial_r)^{\ell}\partial_\varphi^{2-\ell}R(r,\varphi)|^2&+|r\partial_r R(r,\varphi)|^2 \\&+|\partial_\varphi R(r,\varphi)|^2+ |R(r,\varphi)|^2\Bigr)\,r\,dr\,d\varphi<\infty.
\end{aligned}
$$
This together with Sobolev embedding theorem implies $$|x|^{-2+2\epsilon}R(x,\bar x)= O(1),\quad |x|^{-1+2\epsilon}\partial_{x}R(x,\bar x)= O(1), \quad |x|^{-1+2\epsilon}\partial_{\bar x}R(x,\bar x)= O(1).$$
The proof is complete. 
\end{proof}

Let $u\in L^2(X)$ and $v\in L^2(X)$ be such that  $\Delta^*u\in L^2(X)$ and $\Delta^*v\in L^2(X)$ (with differentiation understood in the sense of distributions) and bounded everywhere except possibly for $P_k$. Consider the form $$
\mathsf q[u,v]:=(\Delta^* u,v)-(u,\Delta^* v);
$$
here and elsewhere $(\cdot,\cdot)$ stands for the inner product in ${L^2(X)}$.  By Lemma~\ref{expansion} we have~\eqref{EU} and
\begin{equation}\label{Ev}
v({x},\bar {x})=c_{-1}\bar {x}^{-1} +d_{-1}{x}^{-1}+  c_0\ln|{x}|+ d_0+ c_1\bar {x} +d_1 {x}+ \tilde R({x},\bar {x}).
\end{equation}
The Stokes theorem  implies $$
\mathsf q[u,v]=\lim_{\epsilon\to0+}\int_{X\setminus\{x:|x|<\epsilon\}}(\Delta^* u\bar v-u\overline{\Delta^* v})\,d\textit{Vol}=2i\lim_{\epsilon\to 0+}\oint_{|{x}|=\epsilon}(\partial_{{x}}u)\bar v\,d {{x}}+u(\partial_{\bar {x}} \bar v)d\bar {x}.
$$
Now simple calculation in the right hand side allows to express $\mathsf q[u,v]$ in terms of coefficients in~\eqref{EU} and~\eqref{Ev}  as follows:
\begin{equation}\label{q}
\mathsf q[u,v]=4\pi(-a_{-1}\bar d_1-b_{-1}\bar  c_1-b_0\bar c_0/2+ a_0\bar d_0/2+b_1\bar c_{-1}+a_1\bar d_{-1}).
\end{equation}

Recall that $\Delta$ stands for the Friedrichs extension of the Laplace-Beltrami operator $\Delta^*$ on $(X,f^*\mathsf m)$. As is known, for the domain $\mathscr D$ of $\Delta$ we have  $\mathscr D\subset H^1(X)$. The embedding $H^1(X) \hookrightarrow L^2(X)$ is compact and  the spectrum of $\Delta$ is discrete. Thanks to $|u(p)|\leq \|u; H^1(X)\|$, $p\in X$, the functions in the domain $\mathscr D$ are bounded and thus for any $u\in\mathscr D$ the assertion of  Lemma~\ref{expansion}  is valid with $a_{-1}=b_{-1}=a_{0}=0$.

Let $\chi\in C^\infty_c(X)$ be a cut-off function supported in the neighbourhood $|x|<2\delta$ of $P_k$ and  such that $\chi(|x|)=1$ for $|x|<\delta$, where $\delta$ is small.  We denote the spectrum of $\Delta$ by $\sigma(\Delta)$ and introduce
$$
Y(\lambda)=\chi x^{-1} -(\Delta-\lambda)^{-1}(\Delta^*-\lambda)\chi x^{-1}, \quad \lambda\notin \sigma(\Delta),
$$
where the function $\chi x^{-1}$ is extended from the support of $\chi$ to $X$ by zero.
It is clear that $Y(\lambda)\in L^2(X)$, $Y(\lambda)\not =0$ as $\chi x^{-1}\notin\mathscr D$, and $(\Delta^*-\lambda)Y(\lambda)=0$. By  Lemma~\ref{expansion}  we  have
\begin{equation}\label{expY}
Y(x,\bar x; \lambda)=x^{-1}+c(\lambda)+a(\lambda)\bar x+b(\lambda) x+ O(|x|^{2-\epsilon}),\quad x\to0,\quad \epsilon>0.
\end{equation}

In the remaining part of this section we prove some results that previously appeared in the context of flat conical metrics~\cite{HK,HKK}. 

\begin{lemma}\label{b(lambda)}  The function $Y(\lambda)$ and the coefficient $b(\lambda)$ in~\eqref{expY}  are analytic functions of $\lambda$ in $\Bbb C\setminus\sigma(\Delta)$ and in a neighbourhood of zero. Besides, we have
 \begin{equation}\label{refl}
 4\pi\frac{d}{d\lambda}b(\lambda)=\bigl(Y(\lambda),\overline{Y(\lambda)}\bigr).
 \end{equation}
 \end{lemma}
\begin{proof}  
Since $\ker \Delta=\operatorname{span}\{1\}$, in a neighbourhood of $\lambda=0$ the resolvent admits the representation
$$(\Delta-\lambda)^{-1}f=\lambda^{-1}(f, {\textit{Vol}}(X)^{-1})+R(\lambda)\bigl(f-(f, {\textit{Vol}}(X)^{-1})),$$
where $R(\lambda)$ is a holomorphic operator function with values in the space of bounded operators in $L^2(X)$.
Observe that
$$
\bigl((\Delta^*-\lambda)\chi x^{-1},1\bigr)=\mathsf q[\chi x^{-1},1]-\lambda(\chi x^{-1},1)=-\lambda(\chi x^{-1},1);
$$
therefore $\lambda\mapsto Y(\lambda)\in L^2(X)$ is  holomorphic in a neighbourhood of zero.
Thanks to
$$
b(\lambda)=\frac{1}{4\pi}\mathsf q[Y(\lambda), \chi \bar x^{-1}]=\frac{1}{2\pi}\bigl(Y(\lambda), (\Delta^*-\bar\lambda)\chi \bar x^{-1}\bigr)
$$
the coefficient $b(\lambda)$ is holomorphic together with $Y(\lambda)$.

We obtain the equality~\eqref{refl} as follows:
$$
4\pi \frac{d}{d\lambda} b(\lambda)=\mathsf q\Bigl[\frac{d}{d \lambda} Y(\lambda),\chi \bar x^{-1}\Bigr]
= \mathsf q\Bigl[(\Delta-\lambda)^{-1}\chi x^{-1}-(\Delta-\lambda)^{-2}(\Delta^*-\lambda)\chi  x^{-1},\chi \bar x^{-1}\Bigr]
$$
$$
= \mathsf q\Bigl[(\Delta-\lambda)^{-1}Y(\lambda),\overline {Y(\lambda)}\Bigr]
= \bigl( (\Delta^*-\lambda)(\Delta-\lambda)^{-1}Y(\lambda),\overline{ Y(\lambda)}\bigr)
=\bigl(Y(\lambda),\overline{Y(\lambda)}\bigr).
$$
One can also show that the coefficients $c(\lambda)
$ and $a(\lambda)=\overline{a(\bar\lambda)}
$ in~\eqref{expY} are holomorphic in a neighbourhood of zero. Moreover,    $4\pi \frac {d}{d\lambda} a(\lambda)=\bigl(Y(\lambda),Y(\bar\lambda)\bigr)$.

\end{proof}
\begin{lemma}\label{Eigenf} Let $\{\Phi_j\}_{j=0}^\infty$ be a complete set of real normalized eigenfunctions of $\Delta$ and let $\lambda_j$ be the corresponding eigenvalues, i.e. $\Delta\Phi_j=\lambda_j\Phi_j$, $\Phi_j=\overline{\Phi_j}$, and $\|\Phi_j; L^2(X)\|=1$. Then for the coefficients $a_j$ and  $b_j=\bar a_j $  in the asymptotic
\begin{equation}\label{Phi_j}
\Phi_j(x,\bar x)=c_j+ a_j\bar x+ b_j x+O(|x|^{2-\epsilon}),\quad x\to 0, \quad \epsilon>0,
\end{equation}
we have
\begin{equation}\label{a_j}
16\pi^2\sum_{j=0}^\infty \frac{b_j^2}{(\lambda_j-\lambda)^2}=\bigl(Y(\lambda),\overline{Y(\lambda)}\bigr),
\end{equation}
where  the series is absolutely convergent.
\end{lemma}
\begin{proof} The asymptotic~\eqref{Phi_j} for $\Phi_j\in\mathscr D$  follows from Lemma~\ref{expansion}. Starting from the eigenfunction expansion of $Y(\lambda)$ we obtain
$$
Y(\lambda)=\sum_{j=0}^\infty \bigl(Y(\lambda),\Phi_j\bigr)\Phi_j=\sum_{j=0}^\infty \frac {\bigl(Y(\lambda),(\Delta-\bar\lambda)\Phi_j\bigr)}{\lambda_j-\lambda}\Phi_j
=\sum_{j=0}^\infty \frac {\mathsf q[Y(\lambda),\Phi_j]}{\lambda_j-\lambda}\Phi_j.
$$
This together with~\eqref{q} and $b_j=\bar a_j$ gives
\begin{equation}\label{star}
Y(\lambda)=-4\pi\sum_{j=0}^\infty\frac{b_j}{\lambda_j-\lambda}\Phi_j.
\end{equation}
As a consequence, the series in~\eqref{a_j} is absolutely convergent and
 $$
  \sum_{j=0}^\infty\frac{|b_j|^2}{|\lambda_j-\lambda|^2}=\frac{1}{16\pi^2}\|Y(\lambda); L^2(X)\|^2.
 $$
Finally, we obtain~\eqref{a_j} substituting the expression~\eqref{star} and its conjugate into the inner product $\bigl(Y(\lambda), \overline{Y(\lambda)}\bigl)$.
\end{proof}

\subsection{Explicit calculation of $b(0)$ and $b(-\infty)$}\label{sect}
In this subsection we study the behaviour of the coefficient $b(\lambda)$ from (\ref{expY})  as $\lambda\to-\infty$ and obtain explicit formulas for  $b(-\infty)=\lim_{\lambda\to-\infty}b(\lambda)$ and $b(0)$. Let us emphasize that the choice of the local parameter $x$ in a vicinity of  $P_k\in X$ is a part of definition of the coefficients $a(\lambda), b(\lambda)$, and $c(\lambda)$ in \eqref{expY}.

\begin{lemma} \label{binfty} As $\lambda\to-\infty$ for the coefficient $b(\lambda)$ in~\eqref{expY} we have
\begin{equation*}
b(\lambda)=\frac{1}{2}\frac{\bar z_k}{1+|z_k|^2}+O(|\lambda|^{-\infty}).
\end{equation*}

\end{lemma}
\begin{proof}
{\bf Case 1.} Consider the meromorphic function $f: X={\mathbb C}P^1\to {\mathbb C}P^1$  given by $z=f(w)=w^2$; the critical values of $f$ are $z_1=0$ and $z_2=\infty$. Clearly, $w$ coincides with the distinguished parameter $x=\sqrt{z-z_1}$, the metric  $f^*\mathsf m$ and the Laplace-Beltrami operator $\Delta^*$ are given by~\eqref{lap}, where $z_k=z_1=0$.
Introduce the geodesic polar coordinates $(r,\varphi)$ on $({\mathbb C}P^1, f^*\mathsf m)$ with center at $\infty\in\Bbb CP^1$ by setting $\varphi=2\arg w\in [0,4\pi)$ and $\cot(r/2)=|w|^2$, ${r}\in[0,\pi]$.  In the coordinates $(r,\varphi)$ we have
$$ \Delta^* = -\partial_r^2-\cot r \,{\partial_r}-(\sin r)^{-2}\partial_\varphi^2$$
 and the function $Y$ with asymptotic~\eqref{expY} can be found by separation of variables.  Namely, we seek for $Y$ of the form 
\begin{equation}\label{Y0}
Y(r, \varphi; \lambda)=R(\cos r)e^{-i \varphi/2}.
\end{equation}
For~\eqref{Y0} the equation $(\Delta^*-\lambda)Y=0$  reduces to  
the Legendre equation
\begin{equation}\label{Leg}
(1-t^2)R''(t)-2t R'(t)+\left[\lambda-\left(\frac{1}{2}\right)^2\frac{1}{1-t^2}\right]R(t)=0
\end{equation}
on the line segment  $[-1,1]$, where  $t=\cos r$ and the solution  $R(t)$   should be
 bounded at $t=1$ and have the asymptotic $R(\cos {r})=\sqrt{\tan ({{r}}/{2})}+O(1)$ as ${r}\to \pi$ (i.e. as $x\to 0$).
Observe that $R(t)=-\frac{2}{\sqrt{\pi}\cos(\nu\pi)}Q^{1/2}_\nu(t)$, where $Q^{1/2}_\nu$ is  the associated Legendre function $$Q^{1/2}_\nu(\cos {r})=-\left(\frac{\pi}{2\sin{r}}\right)^{1/2}\sin \bigl((\nu+1/2){r}\bigr)$$ satisfying~\eqref{Leg} with $\lambda=\nu(\nu+1)$; see \cite[p. 359, f-la 14.5.13]{NIST}.
This together with~\eqref{Y0}  gives
\begin{equation}\label{spec}
Y(r, \varphi; \lambda)=\frac{1}{w}\left(\frac{\cos (\nu{r})}{\cos(\nu\pi)}+\frac{\sin (\nu{r})}{\cos(\nu\pi)}|w|^2\right).
\end{equation}
Since $w=x$ and 
$$
\frac{\cos \nu{r}}{\cos\nu\pi}=1-\nu\tan(\nu\pi) \frac{{r}-\pi}{\cot ({r}/2)}|x|^2+O(|x|^4)=1+2\nu\tan(\nu\pi)|x|^2+o(|x|^2)\ \text{ as } x\to 0,
$$
 we conclude that  in the asymptotic~\eqref{expY} of~\eqref{spec} we have $b(\lambda)\equiv 0$ (and also $c(\lambda)\equiv 0$ and $a(\lambda)=(1+2\nu)\tan(\nu\pi)$).

{\bf Case 2.} Consider $\hat f: {\mathbb C}P^1\to {\mathbb C}P^1$ given by $z=\hat f(w)=\frac{w^2+z_1}{1-\bar z_k w^2}$; the critical values of $\hat f$ are $z_1$ and $-1/\bar z_1$. As in the first case, the metric $\hat f^*\mathsf m$ has two antipodal $4\pi$-conical points (at  $w=0$ and $w=\infty$). However the distinguished parameter $x=\sqrt{z-z_1}$ does not coincide with $w$ if $z_1\neq 0$. As a consequence, the corresponding  function $Y$ and the coefficient $b(\lambda)$ in its asymptotic \eqref{expY}  can be different from those obtained in Case 1.

We notice that the isometry $z\mapsto \frac{\alpha z+\beta}{-\overline{\beta} z+\alpha}$
of the base $({\mathbb C}P^1, \mathsf m)$ of a ramified covering $f: X\to {
\mathbb C}P^1$ can be lifted to the corresponding isometry of $(X, f^*\mathsf m)$ and the latter commutes with $\Delta^*$. Take the isometry
 $z\mapsto \frac{z-z_1}{\bar z_1 z-1}$ of  $({\mathbb C}P^1, \mathsf m)$ sending $z_1$ to $0$ and let $J$ be its lift to $({\mathbb C}P^1, \hat f^*\mathsf m)$. We transform $Y$ from \eqref{spec} by $J$ and renormalize 
\begin{equation}\label{hat}
\widehat Y= (1+|z_1|^2)^{-1/2}  Y\circ J.
\end{equation}
It is straightforward to check that $\widehat Y$ has the asymptotic \eqref{expY} in the distinguished local parameter $x=\sqrt{z-z_1}$ and for the corresponding coefficient $b(\lambda)$ we have
\begin{equation*}\label{as-infty}
b(\lambda)=\frac{1}{2}\frac{\bar z_1}{1+|z_1|^2}.\end{equation*}
{\bf Case 3.} Finally, consider the general case.  Let  $X$ be a compact Riemann surface and let $f: X\to {\Bbb  C}P^1$ be a meromorphic function with simple poles and simple critical points $P_1,\dots, P_M$. 

Consider, for instance, the critical point $P_1$.
The function $f$  has the same critical value $z_1$ as the function $\hat f$ from Case 2. Small vicinities $ U( P_1)$ and $\widehat U(\widehat P_1)$ of the corresponding critical points $P_1\in X$ and $\widehat P_1\in\widehat X= \Bbb CP^1$ are isometric. In the local parameter $x=\sqrt{z-z_1}$ (which is the distinguished one  for both  $X$ and  $\widehat X$) the differential expressions $\Delta^*$ and $\widehat \Delta^*$
are the same. 

Let $\rho$ be a smooth cut-off function on $\widehat X$  such that $\rho$ is supported inside $\widehat U(\widehat P_1)$, $\rho\equiv1$  in a vicinity of $\widehat P_1$, and $\rho$ depends only on the distance to $\widehat P_1$. We identify $P_1$ and $\widehat P_1$ as well as $U(P_1)$ and $\widehat U(\widehat P_1)$ and then  extend the functions $\rho \widehat Y$ and $(\Delta^*-\lambda)\rho \widehat Y=[\widehat\Delta^*,\rho]\widehat Y$ from 
$U(P_1)\equiv \widehat U(\widehat P_1)$ to $X$ by zero; here $\widehat Y$ is the function~\eqref{hat}  on $\widehat X=\Bbb CP^1$. Clearly, $[\widehat\Delta^*,\rho]\widehat Y\in L_2(X)$
and therefore $(\Delta -\lambda)^{-1}(\Delta^*-\lambda)\rho \widehat Y$ makes sense. 

Now we represent the function $Y$ on $X$ corresponding to $P_1$ in the form 
\begin{equation}\label{tilde}
 Y(\lambda)=\rho \widehat Y(\lambda)+ ( \Delta -\lambda)^{-1}(\Delta^*-\lambda)\rho \widehat Y(\lambda).\end{equation}
Let $b(\lambda)$ be the coefficient from the asymptotic \eqref{expY} of $Y$. 
We have
\begin{equation}\label{star1}
\begin{aligned}
4\pi\left(b(\lambda)-\frac{1}{2}\frac{\bar z_1}{1+|z_1|^2}\right)=\mathsf q\left[Y(\lambda)-\rho\widehat Y(\lambda), \overline{Y(\lambda)}\right]\\=\left((\Delta^*-\lambda)(Y(\lambda)-\rho\widehat Y(\lambda)),\overline{Y(\lambda)}\right)\\
=-\left([\Delta^*,\rho]\widehat Y(\lambda),\overline{\rho \widehat Y(\lambda)+ ( \Delta -\lambda)^{-1}[\Delta^*,\rho] \widehat Y(\lambda)}\right),
\end{aligned}
\end{equation}
where the right hand side goes to zero like  $O(|\lambda|^{-\infty})$  as $\lambda\to-\infty$. Indeed, from the explicit formulas \eqref{spec} and \eqref{hat} one immediately sees that $||[\Delta^*, \rho]Y; L^2(X)||=O(|\lambda|^{-\infty})$ and that $|Y(\lambda)|=O(|\lambda|^{-\infty})$ uniformly on the support of $[\Delta^*, \rho]Y$ as $\lambda \to -\infty$  (i.e.  as $\Im \nu \to +\infty$, where $\lambda=\nu(\nu+1)$). This together with~\eqref{star1} completes the proof.
\end{proof}

In order to find the value $b(0)$ corresponding to a conical point $P_k$ we need to construct a  (unique up to addition of a constant) harmonic function $Y$  bounded everywhere on $X$ except for the point $P_k$, where $Y(x,\bar x;0)=\frac{1}{x}+O(1)$ in the distinguished local parameter $x=\sqrt{z-z_k}$ (cf.~\eqref{expY}). Such a function was explicitly constructed in~ \cite{HK,HKK} using the canonical meromorphic bidifferential $W(\,\cdot\,, \,\cdot\,)$  (also known as the Bergman bidiffential or the Bergman kernel) on $X$. This leads to an explicit expression for the coefficient $b(0)$ in the asymptotic expansion~\eqref{expY} of $Y$, which was obtained as a part of Proposition 6 in \cite{HKK}. To formulate the result we need some preliminaries.

Chose a marking for the Riemann surface $X$, i.e. a canonical basis $a_1, b_1, \dots, a_g, b_g$ of
 $H_1(X, {\mathbf Z})$. Let $\{v_1, \dots,
v_g\}$ be the basis of holomorphic differentials on $X$ normalized via
$$\int_{a_\ell}v_m=\delta_{\ell m},$$
where $\delta_{\ell m}$ is the Kronecker delta. Introduce the matrix $\mathbb B=(\mathbb B_{\ell m})$ of $b$-periods of the marked Riemann surface $X$  with entries ${\mathbb B}_{\ell m}=\int_{b_\ell}v_m$.
Let $W(\,\cdot\,,\,\cdot\,)$ be the canonical meromorphic bidifferential on $X\times X$ with properties $$W(P,Q)=W(Q, P),\quad
\int_{a_\ell}W(\,\cdot\,, P)=0,\quad
\int_{b_m}W(\,\cdot\,, P)=2\pi iv_m(P).$$
The bidifferential $W$ has the only double pole along the diagonal $P=Q$. In any holomorphic local parameter $x(P)$ one has the asymptotics
\begin{equation}\label{funH}W(x(P), x(Q))=\left(\frac{1}{(x(P)-x(Q))^2}+H(x(P), x(Q))\right)dx(P)dx(Q),\end{equation}
$$H(x(P), x(Q))=\frac{1}{6}S(x(P))+O(x(P)-x(Q)),$$
as $Q\to P$, where $S_B(\cdot)$ is the Bergman projective connection.

Consider the Schiffer bidifferential 
$${\cal S}(P, Q)=W(P, Q)-\pi\sum_{\ell, m}(\Im {\mathbb B})^{-1}_{\ell m}v_\ell(P)v_m(Q).$$
The Schiffer projective connection, $S_{Sch}$, is defined via the asymptotic expansion
$${\cal S}(x(P), x(Q))=\left(\frac{1}{(x(P)-x(Q))^2}+\frac{1}{6}S_{Sch}(x(P))+O(x(P)-x(Q))\right)dx(P)dx(Q).$$
One has the equality
$$S_{Sch}(x)=S_B(x)-6\pi \sum_{\ell, m}(\Im {\mathbb B})^{-1}_{\ell m}v_\ell(x)v_m(x).$$
In contrast to the canonical meromorphic differential and the Bergman projective connection, the Schiffer bidifferential and the Schiffer projective connection are independent of the marking of the Riemann surface $X$.
Let us also emphasize that the value of a projective connection at a point of a Riemann surface depends on the choice of the local holomorphic parameter at this point. Now we are in position to formulate the needed result from \cite[Prop. 6]{HKK}.
\begin{lemma}\label{b-value} We have
$$
b(0)=-\frac{1}{6}S_{Sch}(x)\upharpoonright_{x=0},
$$
where $x$ is the distinguished local parameter $x=\sqrt{z-z_k}$ near the point $P_k$. 
\end{lemma}
\begin{proof} We only notice that in~\cite[Prop. 6]{HKK} $Y$ is denoted by $f_1$ (see \cite[f-la (4.7)]{HKK}) and $b(0)$ is denoted by $S^{hh}_{\frac {1} {2} \frac {1}{2}}(0)$ (see \cite[f-la (4.6)]{HKK}). 
\end{proof}

\section{Perturbation of conical singularities } \label{PCS}

Pick a regular point $z_0\in \Bbb C$ such that $z_1,\dots,z_M$ are (end points but) not internal points of the line segments  $[z_0,z_k]$, $k=1,\dots,M$. Consider the union $\mathsf U=\cup_{k=1}^M[z_0,z_k]$. The complement $X\setminus f^{-1}(\mathsf U)$ of the preimage  $f^{-1}(\mathsf U)$ in $X$ has $N$
connected components ($N$ sheets of the covering) and $f$ is a biholomorphic isometry from each of these components equipped with  metric $f^*\mathsf m$ to  $\Bbb CP^1\setminus \mathsf U$ equipped with the standard   metric~\eqref{m}.
 Thus the  Riemann manifold $(X, f^*\mathsf m)$ is isometric to the one obtained by  gluing $N$ copies of the Riemann sphere $(\Bbb CP^1,\mathsf m)$ along the cuts $\mathsf U$ in accordance with a certain gluing scheme.
By perturbation of the conical singularity at $P_k$  we mean a small  shift of the end $z_k$ of the cut $[z_0,z_k]$ on those two copies of the Riemann sphere $(\Bbb CP^1,\mathsf m)$ that  produce $4\pi$-conical angle at $P_k$ after gluing along  $[z_0,z_k]$.

Let $\varrho\in C^\infty_0(\Bbb R)$ be a cut-off function  such that $\varrho(r)=1$ for $x<\epsilon$ and $\varrho(r)=0$ for $r>2\epsilon$, where $\epsilon$ is small. Consider the selfdiffeomorphism
$$
\phi_w(z,\bar z)=z+\varrho(|z-z_k|)w
$$
of the Riemann sphere  $\Bbb CP^1$, where $w\in \Bbb C$ and $|w|$ is small. On two copies of the Riemann sphere (on those two that   produce the conical singularity at $P_k$ after  gluing along  $[z_0,z_k]$) we shift $z_k$ to $z_k+w$ by applying  $\phi_w$. 
We assume that the support of $\varrho$ and the value $|w|$ are so small that only $[z_0,z_k]$ and no other cuts are affected by $\phi_w$. 
 In this section we  consider the perturbed manifold as  $N$ copies of the Riemann sphere $\Bbb CP^1$ glued along the (unperturbed) cuts $\mathsf U$, however  $N-2$ copies are endowed with metric $\mathsf m$ and $2$ certain  copies (mutually glued along $[z_0,z_k]$) are endowed with pullback $\phi_w^*\mathsf m$ of $\mathsf m$ by $\phi_w$.

Let  $(X,f_w^*\mathsf m)$ stand for the perturbed manifold, where $f_w: X\to \Bbb CP^1$ is the  meromorphic function with  critical values $z_1,\dots,z_{k-1}, z_k+w,z_{k+1},\dots,z_M$. By $\Delta_w$ we denote the Friedrichs extension of Laplace-Beltrami operator on $(X,f_w^*\mathsf m)$ and consider $\Delta_w$ as a perturbation of $\Delta_0$ on $(X,f^*\mathsf m)$.

For the matrix representation of the pullback  $\phi_w^*\mathsf m$ of the metric $\mathsf m$ in~\eqref{m} by $\phi_w$ we have
$$
[\phi_w^*\mathsf m](z,\bar z)=\frac{4}{(1+|z+\varrho(|z-z_k|)w|^2)^2} \left(\phi'_w(z,\bar z)\right)^*{\phi'_w(z,\bar z)},
$$
where
$$
\phi'_w(z,\bar z)=\Id+\frac{\varrho'(|z-z_k|)}{2|z-z_k|}\left[
\begin{array}{cc}
   w({\bar z-\bar z_k})  & w(z-z_k)   \\

  \bar w({\bar z-\bar z_k})    & \bar w(z-z_k)
\end{array}
\right]
$$
is the Jacobian matrix;  i.e. the pullback is  given by
$$
\phi_w^*\mathsf m=\frac{1}{2}[d\bar  z \ d  z][\phi_w^*\mathsf m] [dz\ d\bar z]^T.
$$

Clearly, on $\Bbb CP^1$ we have  $ \Delta_0=-\frac{(1+|z|^2)^2} {4} 4\partial_{\bar z}\partial_z$.
A straightforward calculation also shows
\begin{equation}\label{PertDelta}
\begin{aligned}
\Delta_w &-\Delta_0=\Bigl(\frac{2\varrho(|z-z_k|)(z\bar w+\bar z w)}{1+|z|^2} -\frac{\varrho'(|z-z_k|)}{2|z-z_k|}\bigl(w( {\bar z-\bar z_k} )+\bar w( z-z_k )\bigr)\Bigr)\Delta_{0} \\
&+\frac{({1+|z|^2})^2}{4}\left(2\partial_z\frac{\varrho'(|z-z_k|)}{|z-z_k|}w(z-z_k)\partial_z +2\partial_{\bar z} \frac{\varrho'(|z-z_k|)}{|z-z_k|}{\bar w(\bar z-\bar z_k)}\partial_{\bar z}\right)
\\&+O(|w|^2),
\end{aligned}
\end{equation}
where $O(|w|^2)$ stands for a second order operator with smooth coefficients  supported on $\supp \varrho'(|z-z_k|)$ and uniformly bounded by $C|w|^2$.

Notice that the domain $\mathscr D$ of $\Delta_w$ does not depend on $w$. Consider $\mathscr D$ as a Hilbert space endowed with  graph norm of $\Delta_0$.
 Let $\lambda$ be an eigenvalue of $\Delta_0$ of multiplicity $m$. Let $\Gamma$ be a closed curve enclosing $\lambda$ but no other eigenvalues of $\Delta_0$. Then $$\|(\Delta_0-\xi)^{-1};\mathcal B(L^2;\mathscr D)\|\leq c\|(\Delta_0-\xi)^{-1};\mathcal B(L^2)\|\leq C$$ uniformly in $\xi\in \Gamma$. The resolvent $(\Delta_w-\xi)^{-1}$ exists for all $\xi\in\Gamma$ provided $|w|$ is so small that  $\|(\Delta_w-\Delta_0); \mathcal B(\mathscr D; L^2)\|<1/C$.
Moreover,  $\|(\Delta_w-\xi)^{-1}-(\Delta_0-\xi)^{-1}; \mathcal B(L^2,\mathscr D)\|\to 0$ as $|w|\to 0$ uniformly in $\xi\in\Gamma$.  Therefore the total projection $P_w$ for the eigenvalues of $\Delta_w$ lying inside  $\Gamma$ is given by
$$
P_w=-\frac{1}{2\pi i}\oint_\Gamma (\Delta_w-\xi)^{-1}\,d\xi.
$$
 The continuity of $P_w$ implies that $\dim P_w L^2=\dim P_0 L^2=m$, i.e. the sum of multiplicities of the eigenvalues of $\Delta_w$ lying inside $\Gamma$ is equal to $m$ (provided $|w|$ is small); these eigenvalues are said to form the $\lambda$-group~\cite{Kato}.

 \begin{lemma}\label{L1}
  Consider the power sum symmetric polynomial  $p_n(w)=\sum_{j=1}^m\lambda^n_j(w)$ of degree $n=0,1,2,\dots$ for the  $\lambda$-group $\lambda_1,\dots,\lambda_m$. As $w\to 0$ we obtain
   $$
  p_n(w)=m\lambda^n+n\lambda^{n-1}(Aw+B\bar w)+O(|w|^2),
  $$
  where $\lambda=\lambda_j(0)$, $j=1,\dots,m$,  is the eigenvalue  of $\Delta_0$ of multiplicity $m$. Moreover, the coefficients $A$ and $B$ are given by
  \begin{equation}\label{AB}
  \begin{aligned}
  A=2i \lim_{\epsilon\to0+}  \sum_{j=1}^m
  \oint_{|z-z_k|=\epsilon} (\partial_z\Phi_j)^2\,dz,
  \quad B= -2i\lim_{\epsilon\to0+}\sum_{j=1}^m 
  \oint_{|z-z_k|=\epsilon} (\partial_{\bar z}\Phi_j)^2\,d\bar z,
  \end{aligned}
  \end{equation}
where integration runs around the conical point at $z_k$ through two spheres $\Bbb CP^1$ glued to each other along the cut $[z_0,z_k]$ and  $\Phi_1,\dots,\Phi_m$ are (real) normalized eigenfunctions of $\Delta_0$ corresponding to the eigenvalue $\lambda$; i.e. $\Phi_j=\overline{\Phi_j}$, $\|\Phi_j; L^2(X)\|=1$, and $\operatorname{span}\{\Phi_1,\dots,\Phi_m\}=P_0L^2(X)$.

 \end{lemma}
 \begin{proof} We have $p_n(w)=\Tr(\Delta_w^nP_w)$. Thus
 $$
\begin{aligned}
p_n(w)-m\lambda^n & =-\frac{1}{2\pi i  }\Tr\oint_\Gamma (\xi^n-\lambda^n)(\Delta_w-\xi)^{-1}\,d\xi
\\
&=-\frac{1}{2\pi i  }\Tr\oint_\Gamma (\xi^n-\lambda^n)\sum_{k=1}^\infty(\Delta_0-\xi)^{-1}\bigl[(\Delta_0-\Delta_w)(\Delta_0-\xi)^{-1}\bigr]^k\,d\xi.
\end{aligned}
$$
Taking into account that $\partial_\xi (\Delta_w-\xi)^{-1}=(\Delta_w-\xi)^{-2}$ and
$$
\Tr \partial_\xi \bigl((\Delta_0-\Delta_w)(\Delta_0-\xi)^{-1}\bigr)^k=k\Tr(\Delta_w-\xi)^{-1}\bigl[(\Delta_0-\Delta_w)(\Delta_0-\xi)^{-1}\bigr]^k
$$
(here we applied the identity $\Tr \mathcal A\mathcal B=\Tr \mathcal B\mathcal A$) we obtain
$$
\begin{aligned}
p_n(w)-m\lambda^n=-\frac{1}{2\pi i  }\Tr\oint_\Gamma (\xi^n-\lambda^n)\sum_{k=1}^\infty\frac{1}{k}\partial_\xi\bigl[(\Delta_0-\Delta_w)(\Delta_0-\xi)^{-1}\bigr]^k\,d\xi
\\
=\frac{1}{2\pi i  }\Tr\oint_\Gamma n\xi^{n-1}\sum_{k=1}^\infty\frac{1}{k}\bigl[(\Delta_0-\Delta_w)(\Delta_0-\xi)^{-1}\bigr]^k\,d\xi
\\
=\frac{1}{2\pi i  }\Tr\oint_\Gamma n(\xi^{n-1}-\lambda^{n-1})\sum_{k=1}^\infty\frac{1}{k}\bigl[(\Delta_0-\Delta_w)(\Delta_0-\xi)^{-1}\bigr]^k\,d\xi
\\
+\frac{1}{2\pi i  }n\lambda^{n-1}\Tr\oint_\Gamma \sum_{k=1}^\infty\frac{1}{k}\bigl[(\Delta_0-\Delta_w)(\Delta_0-\xi)^{-1}\bigr]^k\,d\xi
\\
=\frac{1}{2\pi i  }\Tr\oint_\Gamma n(n-1)\xi^{n-2}\sum_{k=1}^\infty\frac{1}{k(k+1)}(\Delta_0-\Delta_w)\bigl[(\Delta_0-\Delta_w)(\Delta_0-\xi)^{-1}\bigr]^k\,d\xi
\\
+\frac{1}{2\pi i  }n\lambda^{n-1}\Tr\oint_\Gamma \sum_{k=1}^\infty\frac{1}{k}\bigl[(\Delta_0-\Delta_w)(\Delta_0-\xi)^{-1}\bigr]^k\,d\xi
\\
=\frac{1}{2\pi i  }n\lambda^{n-1}\Tr\oint_\Gamma (\Delta_0-\Delta_w)(\Delta_0-\xi)^{-1}\,d\xi+O(|w|^2);
\end{aligned}
$$
here we integrated by parts two times and implemented~\eqref{PertDelta} to estimate the remainder. Thus
 \begin{equation}\label{a1}
 \begin{aligned}
 p_n(w)-m\lambda^n &=n\lambda^{n-1}\Tr(\Delta_w-\Delta_0)P_0+O(|w|^2)\\ &=n\lambda^{n-1}\sum_{j=1}^m \bigl((\Delta_w-\Delta_0)\Phi_j,\Phi_j\bigr)_{L^2(X)}+O(|w|^2).
 \end{aligned}
 \end{equation}

 Thanks to~\eqref{PertDelta} we also obtain
 \begin{equation}\label{a2}
  \bigl((\Delta_w-\Delta_0)\Phi_j,\Phi_j\bigr)_{L^2(X)}=A_j w+B_j\bar w+O(|w|^2),
 \end{equation}
 where
  $$\begin{aligned}
A_j= \int\left [  \frac{4}{({1+|z|^2})^2}\left (\frac{2\varrho(|z-z_k|)\bar z }{1+|z|^2}\right.\right.&\left.-\frac{\varrho'(|z-z_k|)}{2|z-z_k|}( {\bar z-\bar z_k}) \right)\lambda\Phi^2_j
\\
&\left.-2\frac{\varrho'(|z-z_k|)}{|z-z_k|}(z-z_k)(\partial_z\Phi_j)^2\right]\frac{dz\wedge d\bar z}{-2i},
\end{aligned}
$$
$$
\begin{aligned}
B_j= \int\left [   \frac{4}{({1+|z|^2})^2}\left (\frac{2\varrho(|z-z_k|) z }{1+|z|^2}\right.\right. & \left.-\frac{\varrho'(|z-z_k|)}{2|z-z_k|}( { z- z_k}) \right)\lambda\Phi^2_j
\\
& \left.-2\frac{\varrho'(|z-z_k|)}{|z-z_k|}(\bar z-\bar z_k)(\partial_{\bar z}\Phi_j)^2\right]\frac{dz\wedge d\bar z}{-2i};
\end{aligned}
$$
here thanks to  $\varrho$ the integrand is supported near $z_k$ and  integration runs through two spheres glued along the cut $[z_0,z_k]$.
Finally, the Stokes theorem implies
\begin{equation}\label{a3}
\begin{aligned}
A_j=2i \lim_{\epsilon\to0+}\left(\oint_{|z-z_k|=\epsilon} (\partial_z\Phi_j)^2\,dz-\lambda \oint_{|z-z_k|=\epsilon}\Phi_j^2 (1+|z|^2)^{-2}\,d\bar z\right),
\\
B_j=-2i \lim_{\epsilon\to0+} \left(\oint_{|z-z_k|=\epsilon} (\partial_{\bar z}\Phi_j)^2\,d\bar z-\lambda \oint_{|z-z_k|=\epsilon}\Phi_j^2 (1+|z|^2)^{-2}\,d z\right).
\end{aligned}
\end{equation}
Since $\Phi_j(p)\leq C$ for $p\in X$,  the last integrals in both formulas~\eqref{a3} tend to zero as $\epsilon\to 0+$.
The assertion follows from~\eqref{a1},~\eqref{a2}, and~\eqref{a3}.
  \end{proof}

 \begin{lemma}\label{L2} Consider the elementary symmetric polynomials
 $$
 e_n(w)=\sum_{1\leq j_1<j_2<\cdots<j_n\leq m} \lambda_{j_1}(w)\,\lambda_{j_2}(w)\cdots\lambda_{j_n}(w),\quad n=1,\dots,m,
 $$
  for the  $\lambda$-group $\lambda_1,\dots,\lambda_m$. As $w\to 0$ we have
  $$
  e_n(w)=\binom mn \Bigl(\lambda^n+n\lambda^{n-1}(Aw+ B\bar w)\Bigr)+O(|w|^2)
  $$
  with  $A$ and $B$ given in~\eqref{AB}.
 \end{lemma}
 \begin{proof} The proof by induction relies on Lemma~\ref{L1} and the relation $$e_n(w)=\frac{1}{n}\sum_{j=1}^n(-1)^{j-1}e_{n-j}(w) p_j(w),$$ where $e_0(w)=1$. We omit details.
 \end{proof}
\begin{lemma} \label{diff}As $w\to 0$ for the $\lambda$-group $\lambda_1,\dots,\lambda_m$ we have
$$
\sum_{j=1}^m\frac{1}{\bigl(\xi-\lambda_j(w)\bigr)^2}=\frac{m}{(\xi-\lambda)^2}+\frac{2(Aw+B\bar w)}{\bigl(\xi-\lambda\bigr)^3}+O(|w|^2)
$$
with $A$ and $B$ given in~\eqref{AB}.
\end{lemma}
\begin{proof} As is well known,   $$\prod_{j=1}^m(\xi-\lambda_j(w))=\sum_{j=0}^m \xi^{m-j}(-1)^j e_j(w).$$ Notice that
$$
\sum_{j=1}^m\frac{1}{\bigl(\xi-\lambda_j(w)\bigr)^2}=-\partial_\xi\sum_{j=1}^m \frac{1}{\xi-\lambda_j(w)}=-\partial_\xi\frac{\sum_{j=0}^{m-1} (m-j) \xi^{m-j-1}(-1)^j e_{j}(w)}{\sum_{j=0}^m \xi^{m-j}(-1)^j e_j(w)}.
$$
%
%
%



We differentiate the right hand side and use Lemma~\ref{L2} to derive asymptotics of resulting numerator and denominator as $w\to0$. We obtain
$$
\sum_{j=1}^m\frac{1}{\bigl(\xi-\lambda_j(w)\bigr)^2}=m\frac{(\xi-\lambda)^{2m-2}-2(m-1)(\xi-\lambda)^{2m-3}(Aw+B\bar w)+O(|w|^2)}{(\xi-\lambda)^{2m}-2m(\xi-\lambda)^{2m-1}(Aw+B\bar w)+O(|w|^2)}.
$$
This implies the assertion.
\end{proof}

\section{Variation of $\ln\det'\Delta$ due to perturbation of conical singularities}

\begin{proposition}\label{pro} Let $w\in \Bbb C$ correspond to perturbation of the conical singularity at $P_k$   by shifting $z_k$ to $z_k+w$ (see Sec.~\ref{PCS} for details).
Then$$
\partial_w \ln\det'\Delta\upharpoonright_{w=0}=\frac{b(0)-b(-\infty)}{2}, \quad \partial_{\bar w} \ln\det'\Delta\upharpoonright_{w=0}=\frac{\overline{b(0)-b(-\infty)}}{2},
$$

where $b(\lambda)$ is the coefficient in the asymptotic~\eqref{expY} of the special solution $Y(\lambda)\in L^2(X)$ to $(\Delta^*-\lambda)Y(\lambda)=0$ growing near $P_k$ as $x^{-1}$, where $x$ is the distinguished holomorphic parameter $x=\sqrt{z-z_k}$.
\end{proposition}
\begin{proof}

First we recall that only $\lambda$-groups but not single eigenvalues $\lambda_j$ can be differentiated with respect to $w$ or $\bar w$. Similarly, the series $\Tr(\Delta-\lambda)^{-2}=\sum_{j=0}^\infty(\lambda_j-\lambda)^{-2}$ cannot be differentiated term by term, however, thanks to Lemma~\ref{diff} we can always differentiate partial finite sums  corresponding to the $\lambda$-groups. Thus, if summation with respect to $j$ runs through $m$ eigenvalues $\lambda_k=\lambda_{k+1}=\cdots=\lambda_{k+m}$ forming  $\lambda_k$-group, by Lemma~\ref{diff} we obtain
$$
\partial_w \Bigl(\sum_{j}\frac{1}{(\lambda_j-\lambda)^2}\Bigr)\upharpoonright_{w=0}=-\frac{2A}{(\lambda_k-\lambda)^3}.
$$
Let us rewrite the formula~\eqref{AB} for the coefficients $A$ in terms of the local parameter $x$:
\begin{equation}\label{Ax}
A=i\lim_{\epsilon\to0+}\sum_{j}\oint_{|{x}|=\epsilon}\frac{1}{{x}}(\partial_{x}\Phi_j)^2\,d{x}.
\end{equation}
By Lemma~\ref{expansion} the asymptotic~\eqref{Phi_j}  of  $\Phi_j$  can be differentiated, we have $\partial_x\Phi_j=b_j+O(|x|^{1-\epsilon})$ with any $\epsilon>0$ as $x\to 0$. This  together with~\eqref{Ax} implies $A=2\pi\sum_{j} b_j^2$, and therefore
$$
\partial_w \Bigl(\sum_{j}\frac{1}{(\lambda_j-\lambda)^2}\Bigr)\upharpoonright_{w=0}=-4\pi\sum_j \frac{b_j^2}{(\lambda_j-\lambda)^3}.
$$

Now we are ready to compute the partial derivative of zeta function with respect to $w$ at $w=0$. Let $\Gamma_\xi$ be a contour running at a sufficiently small distance $\epsilon>0$ around the cut $(-\infty, \xi]$, starting at   $-\infty+i\epsilon$, and ending at  $-\infty-i\epsilon$. We have
$$\partial_w \zeta(s;\Delta-\xi) \upharpoonright_{w=0}=\frac{1}{2\pi i (s-1)}\int_{\Gamma_\xi}(\lambda-\xi)^{1-s}\partial_w\Tr(\Delta-\lambda)^{-2}\upharpoonright_{w=0}\,d\lambda
$$
$$
=\frac{2i}{  (s-1)}\int_{\Gamma_\xi}(\lambda-\xi)^{1-s} \sum_{j=0}^\infty \frac{b_j^2}{(\lambda_j-\lambda)^3}\,d\lambda.
$$
Thanks to Lemma~\ref{Eigenf} we can integrate by parts to obtain
$$
\partial_w \zeta(s;\Delta-\xi) \upharpoonright_{w=0}=-i\int_{\Gamma_\xi}(\lambda-\xi)^{-s} \sum_{j=0}^\infty \frac{b_j^2}{(\lambda_j-\lambda)^2}\,d\lambda\,.
$$
Now we use the equality~\eqref{a_j} from Lemma~\ref{Eigenf} together with Lemma~\ref{b(lambda)} and arrive at
$$\begin{aligned}
\partial_w \zeta(s;\Delta-\xi) \upharpoonright_{w=0}=\frac{-i}{16\pi^2}\int_{\Gamma_\xi}(\lambda-\xi)^{-s}\bigl(Y(\lambda),\overline{Y(\lambda)}\bigr)\,d\lambda\\=\frac{-i}{4\pi}\int_{\Gamma_\xi}(\lambda-\xi)^{-s}
\frac{d}{d\lambda}\left\{b(\lambda)-b(-\infty)\right\}\,d\lambda,
\end{aligned}
$$
where $b(-\infty)=\lim_{\lambda\to-\infty}b(\lambda)$.
Using Lemma~\ref{binfty} and integrating by parts once again we get
$$
\partial_w \zeta(s;\Delta-\xi) \upharpoonright_{w=0}=\frac{-is}{4\pi}\int_{\Gamma_\xi}(\lambda-\xi)^{-s}\left\{b(\lambda)-b(-\infty)\right\}\,d\lambda.
$$
Since $\lambda\mapsto b(\lambda)$ is holomorphic in $\Bbb C\setminus\sigma(\Delta)$ and in a neighbourhood of zero (see Lemma~\ref{b(lambda)}), the Cauchy Theorem implies
$$
\partial_w \zeta'(0;\Delta) \upharpoonright_{w=0}=\frac{1}{4\pi i}\int_{\Gamma_\xi}(\lambda-\xi)^{-1}b(\lambda)\,d\lambda\upharpoonright_{\xi=0}=\frac{b(-\infty)-b(0)}{2}\,.
$$
Since ${\rm det}'\,\Delta=\exp\{-\zeta'(0)\}$, this completes the  proof.
\end{proof}

Now the explicit formulas for $b(0)$ and $b(-\infty)$ (see Lemmas~\ref{binfty} and~\ref{b-value} in Sec.~\ref{sect}) together with Proposition~\ref{pro} imply the following Theorem.
\begin{theorem}\label{main} Let $X$ be a compact Riemann surface of genus $g\geq 0$ and let $f$ be a meromorphic function on $X$ of degree $N$ with $N$ simple poles and $M=2N+2g-2$ simple critical points $P_1, \dots, P_M$.
Let $z_k=f(P_k)$ be the critical values of $f$. Consider the determinant ${\rm det}' \Delta$ of the (Friedrichs) Laplacian $\Delta$ in the conical metric $f^*{\mathsf m}$ with constant curvature $1$ on $X$ as a function on the moduli space $H_{g, N}(1, \dots, 1)$ of pairs $(X, f)$ with local coordinates $z_1, \dots, z_M$. Then this function satisfies the following system of differential equations
\begin{equation}\label{System}
\frac{\partial \log {\det}'\Delta}{\partial z_k}=-\frac{1}{12}S_{Sch}(x_k)\upharpoonright_{x_k=0}-\frac{1}{4}\frac{\bar z_k}{1+|z_k|^2}, \quad k=1, \dots, M,
\end{equation}
where $x_k(P)=\sqrt{f(P)-f(P_k)}$ is the distinguished local parameter near the critical point $P_k$ and $S_{Sch}$ is the Schiffer projective connection on $X$.
\end{theorem}
The system \eqref{System} admits explicit integration. In \cite{KK} (see also \cite{KK1, KKZ}) it was shown that
the function ${\det}\Im{\mathbb B}\, |\tau|^2$, where $\tau$ is the so called Bergman tau-function on the Hurwitz space $H_{g, N}(1, \dots, 1)$,
satisfies
 $$\frac{\partial \log ({\det} \Im {\mathbb B}\,|\tau|^2)}{\partial z_k}=-\frac{1}{12}S_{Sch}(x_k)\upharpoonright_{x_k=0}, \quad k=1, \dots, M;$$
in genus $0$ the factor ${\det}\Im {\mathbb B}$ should be omitted.
This together with Theorem \ref{main} immediately leads to the main result of the present paper.
\begin{theorem} The  explicit formula
\begin{equation}\label{teorema}
{\det}'\Delta=C\,{\det}\, \Im {\mathbb B} |\tau|^2 \prod_{k=1}^M(1+|z_k|^2)^{-1/4}
\end{equation}
is valid for the determinant of the Friedrichs extension  $\Delta$ of the Laplacian on  $(X,f^*{\mathsf m})$.
\end{theorem}
\begin{remark}   
 We recall that under the linear fractional transformations $z\mapsto \frac{az+b}{cz+d}$, $ad-bc=1$, the function $\tau^2$ transforms as
$$\tau^2\mapsto \tau^2\prod_{k=1}^M(cz_k+d)^{-1/2};$$
see~\cite[Lemma~1]{KKZ}. Notice that under the $\rm{SU}(2)$ transformation $z\mapsto \frac{dz-\bar c}{cz+d}$, $|d|^2+|c|^2=1$, the factor $F=\prod_{k=1}^M(1+|z_k|^2)^{-1/4}$ in \eqref{teorema} 
 transforms as
$$ F\mapsto F\prod_{k=1}^M|cz_k+d|^{1/2}.$$
Thus we see that the right hand side in~\eqref{teorema} is $\rm{SU}(2)$-invariant as it should be due to $\rm{SU}(2)$-invariance of ${\det}'\Delta$.
\end{remark}

\end{document}